\documentclass[11pt]{article}
\usepackage{geometry,amsmath,amsthm,amssymb,latexsym}

\usepackage[svgnames]{xcolor}
\usepackage[colorlinks=true,linkcolor=MidnightBlue,citecolor=MidnightBlue,urlcolor=MidnightBlue,pdfpagelabels=false]{hyperref}

\usepackage{tikz}
\usetikzlibrary{decorations.pathreplacing,calligraphy,patterns}

\usepackage{thmtools}
\theoremstyle{plain}
  \declaretheorem[numberwithin=section]{theorem}
  \declaretheorem[numberlike=theorem]{corollary}
  
  \declaretheorem[numberlike=theorem]{lemma}

\theoremstyle{definition}
  
  \declaretheorem[numberlike=theorem]{example}
  \declaretheorem[numberlike=theorem]{remark}

\newcommand{\assign}{:=}
\newcommand{\rightarrowlim}{\mathop{\rightarrow}\limits}
\newcommand{\textdots}{...}
\newcommand{\tmtt}[1]{\texttt{#1}}

\begin{document}

\title{Partitions with Durfee triangles of fixed size}

\author{
  N. Guru Sharan
  \and
  Armin Straub
}

\date{July 23, 2025}

\maketitle

\begin{abstract}
  A well-studied statistic of an integer partition is the size of its Durfee
  square. In particular, the number $D_k (n)$ of partitions of $n$ with Durfee
  square of fixed size $k$ has a well-known simple rational generating
  function. We study the number $R_k (n)$ of partitions of $n$ with Durfee
  triangle of size $k$ (the largest subpartition with parts $1, 2, \ldots,
  k$). We determine the corresponding generating functions which are rational
  functions of a similar form. Moreover, we explicitly determine the leading
  asymptotic of $R_k (n)$, as $n \rightarrow \infty$.
\end{abstract}

\section{Introduction}

A partition of an integer $n \geq 0$ is an ordered list $\lambda =
(\lambda_1, \lambda_2, \ldots, \lambda_d)$ of integers $\lambda_1 \geq
\lambda_2 \geq \cdots \geq \lambda_d \geq 1$ such that $n =
\lambda_1 + \lambda_2 + \cdots + \lambda_d$. The $\lambda_j$ are called the
parts of $\lambda$. For nice introductions to partitions, we refer to
\cite{andrews-part} and \cite{ae-partitions}. Recall that the Ferrers
board of the partition $\lambda$ consists of $d$ left-aligned rows such that
the $j^{\text{th}}$ row contains $\lambda_j$ boxes. Its Durfee square is the
largest square within the Ferrers board. We say that the Durfee square of
$\lambda$ has size $k$ if its sides have length $k$ (in which case it consists
of $k \times k$ boxes). Equivalently, $k$ is the largest integer such that
$\lambda_j \geq k$ for all $j \in \{ 1, 2, \ldots, k \}$. Let $D_k (n)$
be the number of partitions of $n$ whose Durfee square has size $k$. It is
well-known that
\begin{equation}
  \sum_{n = 0}^{\infty} D_k (n) q^n = \frac{q^{k^2}}{(1 - q)^2 (1 - q^2)^2
  \cdots (1 - q^k)^2} \label{eq:Dk:gf:rat:intro}
\end{equation}
(see, for instance, \cite[p.~28]{andrews-part} or
\cite[Section~8.1]{ae-partitions}).

As illustrated in Figure~\ref{fig:durfee}, one can likewise consider the
\emph{Durfee triangle} of a partition $\lambda$ as the largest right-angled
isosceles triangle within the Ferrers board whose apex is the top-left corner.
We say that the Durfee triangle of $\lambda$ has size $k$ if its horizontal
and vertical sides have length $k$ each. Equivalently, $k$ is the largest
integer such that $\lambda_j > k - j$ for all $j \in \{ 1, 2, \ldots, k \}$.
This notion of Durfee triangle was recently introduced by the first author
\cite{sharan-rook}, who showed that the size of the Durfee triangle of
$\lambda$ equals the maximal number of non-intersecting rooks that can be
placed on the Ferrer's board of $\lambda$. In the present paper, we study the
number $R_k (n)$ of partitions of $n$ whose Durfee triangle has size $k$. Our
first main result is that the generating functions $\mathcal{F}_k (q)$ of $R_k
(n)$ take a form that is similar to the generating functions
\eqref{eq:Dk:gf:rat:intro} for $D_k (n)$, though more complicated and less
explicit.

\begin{theorem}
  \label{thm:Rk:gf:rat:intro}For any fixed positive integer $k$, we have
  \begin{equation}
    \mathcal{F}_k (q) \assign \sum_{n = 0}^{\infty} R_k (n) q^n = \frac{q^{k
    (k + 1) / 2} \varphi_k (q)}{(1 - q) (1 - q^2) \cdots (1 - q^k)}
    \label{eq:Rk:gf:rat:intro}
  \end{equation}
  where
  \begin{equation*}
    \varphi_k (q) = 1 + c_1 q + \cdots + c_{k^2 - 1} q^{k^2 - 1} + (- 1)^{k -
     1} q^{k^2}
  \end{equation*}
  is a polynomial of degree $k^2$ with integer coefficients. Moreover, if $k$
  is odd then $\varphi_k (- 1) = 0$.
\end{theorem}

The initial cases $k \in \{ 1, 2, 3, 4, 5 \}$ were individually considered in
\cite{sharan-rook} and Theorem~\ref{thm:Rk:gf:rat:intro} extends this to all
$k$ in a uniform manner. Theorem~\ref{thm:Rk:gf:rat:intro} is proved in
Section~\ref{sec:Fk} as Theorem~\ref{thm:Fk:denom} (which is stated in terms
of the normalized variation $F_k (q) = q^{- k (k + 1) / 2} \mathcal{F}_k (q)$)
and Lemma~\ref{lem:Fk:denom:odd}. The proof proceeds from an explicit
representation of the generating function $\mathcal{F}_k (q)$ as a multiple
sum which is discussed in Section~\ref{sec:Fk:convolution}. This $q$-multisum
is a convolution sum of a related simpler $q$-series $A_d (q)$ which is
further analyzed in Section~\ref{sec:Ad}.

\begin{figure}[h]
  \centering
  \begin{tikzpicture}[scale=0.8, thick]
    \draw[pattern=north west lines, pattern color=gray!60] (1,0) rectangle (3,-2);

    \foreach \row/\count in {0/6, 1/4, 2/2, 3/1} {
      \foreach \col in {1,...,\count} {
            \draw (\col,-\row) rectangle ++(1,-1);
        }
    }
  \end{tikzpicture}
  \qquad
  \begin{tikzpicture}[scale=0.8, thick]
    \draw[pattern=north west lines, pattern color=gray!60, draw=gray!60] (1,0) -- (5,0) -- (1,-4) -- cycle;

    \foreach \row/\count in {0/6, 1/4, 2/2, 3/1} {
      \foreach \col in {1,...,\count} {
            \draw (\col,-\row) rectangle ++(1,-1);
        }
    }
  \end{tikzpicture}
  
  \caption{\label{fig:durfee}Durfee square and triangle (of size $2$ and $4$)
  of $\lambda = (6, 4, 2, 1)$}
\end{figure}
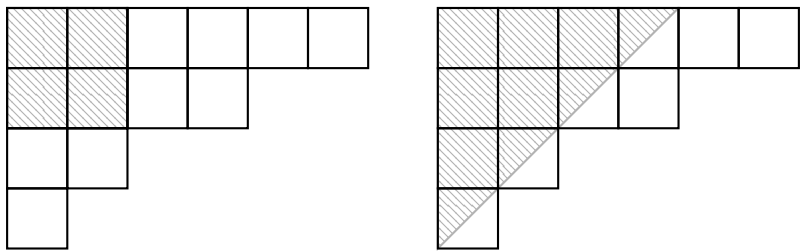

As a consequence of the generating functions \eqref{eq:Rk:gf:rat:intro}, the
sequences $R_k (n)$ are constant recursive, or $C$-finite, of order $k (k + 1)
/ 2$ (and the order can be reduced by $1$ if $k \geq 3$ is odd). This is
discussed in Section~\ref{sec:Rk:formulas}. It further follows (see
Corollary~\ref{cor:Rk:qp}) that the values $R_k (n)$, for fixed $k$ and all $n
> k^2$, are given by a quasi-polynomial of degree $k - 1$ and quasi-period
$\operatorname{lcm} (1, 2, \ldots, k)$. Analogously, Choliy and Sills
\cite{cs-durfeesquare} deduce from \eqref{eq:Dk:gf:rat:intro} that $D_k (n)$
is a quasi-polynomial of degree $2 k - 1$ and quasi-period $\operatorname{lcm} (1, 2,
\ldots, k)$. Moreover, they show that the leading terms are
\begin{equation}
  D_k (n) = \frac{1}{(k!)^2} \frac{n^{2 k - 1}}{(2 k - 1) !} - \frac{1}{2 k!
  (k - 2) !} \frac{n^{2 k - 2}}{(2 k - 2) !} + O (n^{2 k - 3}) .
  \label{eq:Dk:asy:intro}
\end{equation}
Additional terms of the asymptotic expansion as $n \rightarrow \infty$ can be
worked out algorithmically as observed by Sills and Zeilberger
\cite{sz-pmn-quasi}. Establishing the asymptotics for $R_k (n)$ is more
intricate because the generating functions \eqref{eq:Rk:gf:rat:intro} are less
explicit than \eqref{eq:Dk:gf:rat:intro} when working with general $k$. Our
second main result is to determine the leading term in
Section~\ref{sec:Rk:formulas}.

\begin{theorem}
  \label{thm:Rk:asy:intro}Fix a positive integer $k$. As $n \rightarrowlim
  \infty$, we have
  \begin{equation}
    R_k (n) = \frac{2^k}{k!} \frac{n^{k - 1}}{(k - 1) !} + O (n^{k - 2}) .
    \label{eq:Rk:asy:intro}
  \end{equation}
\end{theorem}

It would be desirable to further analyze the generating functions
\eqref{eq:Rk:gf:rat:intro} and, as an application, to determine additional
terms in the asymptotic expansion of $R_k (n)$. Various other avenues for
future work are outlined in the final section.

\section{A convolution sum for \texorpdfstring{$\mathcal{F}_k (q)$}{Fk(q)}}\label{sec:Fk:convolution}

We begin by describing the generating function $\mathcal{F}_k (q)$ in terms of
a convolution sum involving the following simpler $q$-series which we study in
more detail in the next section:
\begin{equation}
  A_d (q) \assign \sum_{k_1 = 0}^{\infty} \sum_{k_2 = 0}^{k_1 + 1} \cdots
  \sum_{k_d = 0}^{k_{d - 1} + 1} q^{k_1 + k_2 + \cdots + k_d} = \sum_{n =
  0}^{\infty} a_d (n) q^n, \label{eq:Aq:def}
\end{equation}
where $d$ is a positive integer. We further set $A_0 (q) = 1$. Throughout, let
\begin{equation*}
  T_k = 1 + 2 + \cdots + k = \frac{1}{2} k (k + 1)
\end{equation*}
be the $k^{\text{th}}$ triangular number.

\begin{lemma}
  \label{lem:F:gf}We have
  \begin{equation}
    \mathcal{F}_k (q) = q^{T_k} \sum_{d = 0}^k q^d A_d (q) A_{k - d} (q) .
    \label{eq:F:gf}
  \end{equation}
\end{lemma}

\begin{proof}
  Let $\lambda = (\lambda_1, \lambda_2, \ldots)$ with $\lambda_1 \geq
  \lambda_2 \geq \ldots$ be a partition with Durfee triangle of size $k$,
  and let $\lambda' = (\lambda_1', \lambda_2', \ldots)$ be its conjugate. This
  implies that $\lambda_j \geq k - j + 1$ (as well as $\lambda_j'
  \geq k - j + 1$) for $j \in \{ 1, 2, \ldots, k \}$. Choose $d \in \{ 0,
  1, \ldots, k \}$ to be maximal such that $\lambda_j > k - j + 1$ for all $j
  \in \{ 1, 2, \ldots, d \}$, and set $m_j = \lambda_j - (k - j + 1)$. In
  other words, $d$ is the number of initial rows in the Ferrers diagram of
  $\lambda$ that are strictly longer than the Durfee triangle, and $m_j
  \geq 1$ is the amount by which the length of the $j^{\text{th}}$ row
  exceeds the Durfee triangle. Note that we have $m_1 \geq 1$, $1
  \leq m_2 \leq m_1 + 1$, {\textdots}, $1 \leq m_d \leq
  m_{d - 1} + 1$.
  
  On the other hand, set $n_j = \lambda_j' - k + j - 1 \geq 0$ for $j \in
  \{ 1, 2, \ldots, k \}$ to be the amount by which the length of the
  $j^{\text{th}}$ column in the Ferrers diagram of $\lambda$ exceeds the
  Durfee triangle. Note that we have $n_1 \geq 0$, $0 \leq n_2
  \leq n_1 + 1$, {\textdots}, $0 \leq n_{k - d} \leq n_{k - d -
  1} + 1$ (and that, by the definition of $d$, we have $n_{k - d + 1} = 0$
  provided that $d \neq 0$). See Figure~\ref{fig:durfee:decomp} for a diagram
  illustrating the situation.
  
  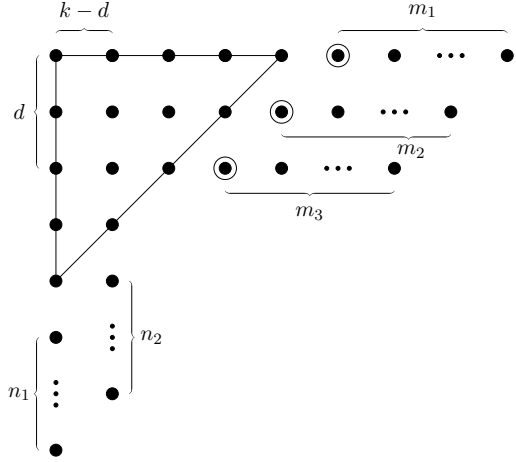
\begin{figure}[h]
 	\centering
    \begin{tikzpicture}[description/.style={fill=white,inner sep=2pt}]
      \useasboundingbox (-5.5,-2) rectangle (1.5,4.5);
      \scope[transform canvas={scale=0.75}]
      \filldraw[black] (-6,1.5) circle (3pt);
      \filldraw[black] (-6,0.5) circle (3pt);
      \filldraw[black] (-6,2.5) circle (3pt);
      \filldraw[black] (-6,3.5) circle (3pt);
      \filldraw[black] (-6,4.5) circle (3pt);
      \filldraw[black] (-5,4.5) circle (3pt);
      \filldraw[black] (-5,2.5) circle (3pt);
      \filldraw[black] (-1.2,2.5) circle (1pt);
      \filldraw[black] (-1,2.5) circle (1pt);
      \filldraw[black] (-0.8,2.5) circle (1pt);
      \filldraw[black] (-3,2.5) circle (3pt);
      \filldraw[black] (-2,2.5) circle (3pt);
      \filldraw[black] (0,2.5) circle (3pt);
      \filldraw[black] (-4,2.5) circle (3pt);
      \filldraw[black] (-5,-0.3) circle (1pt);
      \filldraw[black] (-5,-0.5) circle (1pt);
      \filldraw[black] (-5,-0.7) circle (1pt);
      \filldraw[black] (-5,1.5) circle (3pt);
      \filldraw[black] (-5,0.5) circle (3pt);
      \filldraw[black] (-5,-1.5) circle (3pt);
      \filldraw[black] (-0.2,3.5) circle (1pt);
      \filldraw[black] (0,3.5) circle (1pt);
      \filldraw[black] (0.2,3.5) circle (1pt);
      \filldraw[black] (1,3.5) circle (3pt);
      \filldraw[black] (-1,3.5) circle (3pt);
      \filldraw[black] (-2,3.5) circle (3pt);
      \filldraw[black] (-3,3.5) circle (3pt);
      \filldraw[black] (-5,3.5) circle (3pt);
      \filldraw[black] (-5,4.5) circle (3pt);
      \filldraw[black] (-4,3.5) circle (3pt);
      \filldraw[black] (-4,4.5) circle (3pt);
      \filldraw[black] (-3,4.5) circle (3pt);
      \filldraw[black] (-3,4.5) circle (3pt);
      \filldraw[black] (0.8,4.5) circle (1pt);
      \filldraw[black] (1,4.5) circle (1pt);
      \filldraw[black] (1.2,4.5) circle (1pt);
      \filldraw[black] (2,4.5) circle (3pt);
      \filldraw[black] (-1,4.5) circle (3pt);
      \filldraw[black] (0,4.5) circle (3pt);
      \filldraw[black] (-2,4.5) circle (3pt);
      \filldraw[black] (-6,-1.3) circle (1pt);
      \filldraw[black] (-6,-1.5) circle (1pt);
      \filldraw[black] (-6,-1.7) circle (1pt);
      \filldraw[black] (-6,-0.5) circle (3pt);
      \filldraw[black] (-6,-2.5) circle (3pt);
      \draw (-1,4.5) circle [radius=0.2cm];
      \draw (-2,3.5) circle [radius=0.2cm];
      \draw (-3,2.5) circle [radius=0.2cm];
      \draw (-6,4.5) -- (-2,4.5);
      \draw (-6,4.5) -- (-6,0.5);
      \draw (-2,4.5) -- (-6,0.5);
      \node at (0.5,5.3) {$m_1$};
      \node at (-5.5,5.3) {$k-d$};
      \node at (0.3,2.8) {$m_2$};
      \node at (-1.5,1.7) {$m_3$};
      \node at (-4.3,-0.5) {$n_2$};
      \node at (-6.65,-1.5) {$n_1$};
      \node at (-6.65,3.5) {$d$};
      \draw [decorate, 
      decoration = {calligraphic brace,
          raise=0pt,
          aspect=0.5}] (-1,4.9) --  (2,4.9);
      \draw [decorate, 
      decoration = {calligraphic brace,
          raise=0pt,
          aspect=0.5}] (-6,4.9) --  (-5,4.9);
      \draw [decorate, 
      decoration = {calligraphic brace, mirror, 
          raise=0pt,
          aspect=0.75}] (-2,3.1) --  (1,3.1);
      \draw [decorate, 
      decoration = {calligraphic brace, mirror, 
          raise=0pt,
          aspect=0.5}] (-3,2.1) --  (0,2.1);
      \draw [decorate, 
      decoration = {calligraphic brace,  
          raise=0pt,
          aspect=0.5}] (-4.7,0.5) --  (-4.7,-1.5);
      \draw [decorate, 
      decoration = {calligraphic brace, mirror, 
          raise=0pt,
          aspect=0.5}] (-6.3,-0.5) --  (-6.3,-2.5);
      \draw [decorate, 
      decoration = {calligraphic brace, mirror, 
          raise=0pt,
          aspect=0.5}] (-6.3,4.5) --  (-6.3,2.5);
      \endscope
    \end{tikzpicture}

    \caption{\label{fig:durfee:decomp}Decomposing a partition according to its
    Durfee triangle}
  \end{figure}
  
  The Ferrers diagram of $\lambda$ decomposes into the Durfee triangle, the
  $m_1 + \cdots + m_d$ excess nodes in the first $d$ rows, as well as the $n_1
  + \cdots + n_{k - d}$ excess nodes in the first $k - d$ columns. It follows
  that the generating function $\mathcal{F}_k (q)$ for partitions with Durfee
  triangle of size $k$ is given by
  \begin{equation*}
    q^{T_k} \sum_{d = 0}^k \left[ \sum_{m_1 = 1}^{\infty} \sum_{m_2 = 1}^{m_1
     + 1} \cdots \sum_{m_d = 1}^{m_{d - 1} + 1} q^{m_1 + \cdots + m_d} \right]
     \left[ \sum_{n_1 = 0}^{\infty} \sum_{n_2 = 0}^{n_1 + 1} \cdots \sum_{n_{k
     - d} = 0}^{n_{k - d - 1} + 1} q^{n_1 + \cdots + n_{k - d}} \right] .
  \end{equation*}
  The second bracketed multi-sum is $A_{k - d} (q)$, by the definition of the
  latter. On the other hand, the first bracketed multi-sum can be rewritten as
  \begin{equation*}
    \sum_{m_1 = 1}^{\infty} \sum_{m_2 = 1}^{m_1 + 1} \cdots \sum_{m_d =
     1}^{m_{d - 1} + 1} q^{m_1 + \cdots + m_d} = q^d \sum_{n_1 = 0}^{\infty}
     \sum_{n_2 = 0}^{n_1 + 1} \cdots \sum_{n_d = 0}^{n_{d - 1} + 1} q^{n_1 +
     \cdots + n_d} = q^d A_d (q) .
  \end{equation*}
  Algebraically, this follows by setting $n_i = m_i - 1$; combinatorially,
  this reflects the fact that each of the $d$ rows contains at least one
  excess node and these $d$ first nodes are peeled off into the overall factor
  $q^d$.
  
  In terms of the function $A_d (q)$, we therefore conclude that
  \begin{equation*}
    \mathcal{F}_k (q) = q^{T_k} \sum_{d = 0}^k q^d A_d (q) A_{k - d} (q),
  \end{equation*}
  as claimed.
\end{proof}

\section{The generating function \texorpdfstring{$A_d (q)$}{Ad(q)}}\label{sec:Ad}

Let $A_d (q) = \sum_{n = 0}^{\infty} a_d (n) q^n$ be as defined in
\eqref{eq:Aq:def}. The values of the numbers $a_d (n)$ for small $d$ and $n$
are recorded in the following table:
\begin{equation*}
  \begin{array}{|c|c|c|c|c|c|c|c|c|c|c|c|}
     \hline
     d\backslash n & 0 & 1 & 2 & 3 & 4 & 5 & 6 & 7 & 8 & 9 & 10\\
     \hline
     1 & 1 & 1 & 1 & 1 & 1 & 1 & 1 & 1 & 1 & 1 & 1\\
     \hline
     2 & 1 & 2 & 2 & 3 & 3 & 4 & 4 & 5 & 5 & 6 & 6\\
     \hline
     3 & 1 & 3 & 4 & 6 & 7 & 9 & 11 & 13 & 15 & 18 & 20\\
     \hline
     4 & 1 & 4 & 7 & 11 & 15 & 19 & 25 & 30 & 37 & 44 & 53\\
     \hline
     5 & 1 & 5 & 11 & 19 & 29 & 39 & 53 & 67 & 84 & 103 & 126\\
     \hline
     6 & 1 & 6 & 16 & 31 & 52 & 76 & 107 & 143 & 184 & 233 & 289\\
     \hline
     7 & 1 & 7 & 22 & 48 & 88 & 140 & 207 & 291 & 389 & 508 & 646\\
     \hline
     8 & 1 & 8 & 29 & 71 & 142 & 245 & 384 & 567 & 792 & 1069 & 1401\\
     \hline
   \end{array}
\end{equation*}
By the defining multisum \eqref{eq:Aq:def}, $a_d (n)$ is the number of weak
compositions of $n$ into $d$ parts such that each part does not exceed its
predecessor by more than $1$. In other words, $a_d (n)$ is the number of
tuples $(k_1, k_2, \ldots, k_d)$ of nonnegative integers $k_j$ with $k_1 + k_2
+ \cdots + k_d = n$ such that $k_j \leq k_{j - 1} + 1 $ for $j \in \{ 2,
3, \ldots, d \}$. On the other hand, it follows as in the proof of
Lemma~\ref{lem:F:gf} that $a_d (n)$ equals the number of partitions of $n +
T_d$ into exactly $d$ parts with Durfee triangle of size $d$. In other words,
$q^{T_d} A_d (q)$ can be interpreted as the generating function for partitions
into $d$ parts with Durfee triangle of size $d$.

\begin{example}
  For instance, $a_3 (5) = 9$. The corresponding weak compositions of $5$ into
  $3$ parts such that each part does not exceed its predecessor by more than
  $1$ are
  \begin{equation*}
    (5, 0, 0), (4, 1, 0), (4, 0, 1), (3, 2, 0), (3, 1, 1), (2, 3, 0), (2, 2,
     1), (2, 1, 2), (1, 2, 2) .
  \end{equation*}
  On the other hand, by adding the Durfee triangle $(3, 2, 1)$ componentwise
  to these compositions, we obtain the corresponding partitions of $5 + T_3 =
  11$ into $3$ parts with Durfee triangle of size $3$:
  \begin{equation*}
    (8, 2, 1), (7, 3, 1), (7, 2, 2), (6, 4, 1), (6, 3, 2), (5, 5, 1), (5, 4,
     2), (5, 3, 3), (4, 4, 3) .
  \end{equation*}
\end{example}

\begin{remark}
  The sequences $a_d (n)$ for fixed $d \leq 10$ are recorded in the OEIS
  \cite{sloane-oeis} as the ``number of partitions into $d$ parts such that
  every $i^{\text{th}}$ smallest part (counted with multiplicity) is different
  from $i$''. Here one has to consider partitions of size $n + T_d + d$ and
  the corresponding partitions can be obtained from those of $n + T_d$ into
  $d$ parts with Durfee triangle of size $d$ by simply increasing each part by
  $1$.
\end{remark}

For small values of $d$, we can use the multisum \eqref{eq:Aq:def} to compute
$A_d (q)$ explicitly. For instance, in addition to the trivial $A_1 (q) = 1 /
(1 - q)$, we find
\begin{eqnarray*}
  A_2 (q) & = & \frac{1 + q - q^2}{(1 - q)^2 (1 + q)},\\
  A_3 (q) & = & \frac{1 + 2 q - q^3 - 2 q^4 + q^6}{(1 - q)^3 (1 + q) (1 + q +
  q^2)},\\
  A_4 (q) & = & \frac{(1 + q - q^4) (1 + 2 q - 2 q^4 - q^5 + q^8)}{(1 - q)^4
  (1 + q)^2 (1 + q + q^2) (1 + q^2)} .
\end{eqnarray*}
Firstly, we observe that the denominators of these three rational functions
are $(q ; q)_2$, $(q ; q)_3$ and $(q ; q)_4$, respectively, where $(q ; q)_d$
is the $q$-Pochhammer
\begin{equation}
  (q ; q)_d = \prod_{r = 1}^d (1 - q^r) . \label{eq:qpochhammer}
\end{equation}
The remainder of this section focuses on proving the following result which
shows, in particular, that this observation extends to $A_d (q)$ in general.

\begin{theorem}
  \label{thm:Aq:denom}For all positive integers $d$,
  \begin{equation}
    A_d (q) = \frac{\alpha_d (q)}{(q ; q)_d} \label{eq:Aq:denom}
  \end{equation}
  where $\alpha_d (q) \in 1 + q\mathbb{Z} [q]$ is a polynomial of degree $(d -
  1) d$. Moreover, the coefficient of $q^{(d - 1) d}$ in $\alpha_d (q)$ is $(-
  1)^{d - 1}$.
\end{theorem}

Numerical evidence suggests that Theorem~\ref{thm:Aq:denom} is best possible
in the sense that the quotient on the right-hand side of \eqref{eq:Aq:denom}
is in lowest terms and cannot be further reduced. We have confirmed this for
all $d \leq 20$ by explicitly calculating $A_d (q)$.

In order to prove Theorem~\ref{thm:Aq:denom}, we first give a recursive
characterization of $A_d (q)$. To this end, we introduce the auxiliary
generating functions

\begin{equation}
  \label{eq:Ax:def} A_d (x_1, \ldots, x_d) \assign \sum_{k_1 = 0}^{\infty}
  \sum_{k_2 = 0}^{k_1 + 1} \cdots \sum_{k_d = 0}^{k_{d - 1} + 1} x_1^{k_1}
  \cdots x_d^{k_d} .
\end{equation}

We note that there is no harm in denoting the sums in \eqref{eq:Aq:def} and
\eqref{eq:Ax:def} both with $A_d$ because they are equal in the case $d = 1$.
More generally, we have $A_d (q) = A_d (q, \ldots, q)$, where the right-hand
side features $d$ copies of $q$.

Combined with the base case $A_1 (x_1) = 1 / (1 - x_1)$, the following
recursively determines the series $A_d (x_1, \ldots, x_d)$ for all positive
integers $d$.

\begin{lemma}
  \label{lem:Ax:rec}For integers $d > 1$,
  \begin{equation}
    A_d (x_1, \ldots, x_d) = \frac{A_{d - 1} (x_1, \ldots, x_{d - 1}) - x_d^2
    A_{d - 1} (x_1, \ldots, x_{d - 2}, x_{d - 1} x_d)}{1 - x_d} .
    \label{eq:Ax:rec}
  \end{equation}
\end{lemma}

\begin{proof}
  This follows directly from the definition \eqref{eq:Ax:def} of $A_d (x_1,
  \ldots, x_d)$ by evaluating the innermost geometric series:
  \begin{eqnarray*}
    A_d (x_1, \ldots, x_d) & = & \sum_{k_1 = 0}^{\infty} \sum_{k_2 = 0}^{k_1 +
    1} \cdots \sum_{k_{d - 1} = 0}^{k_{d - 2} + 1} x_1^{k_1} \cdots x_{d -
    1}^{k_{d - 1}} \sum_{k_d = 0}^{k_{d - 1} + 1} x_d^{k_d}\\
    & = & \frac{1}{1 - x_d} \sum_{k_1 = 0}^{\infty} \sum_{k_2 = 0}^{k_1 + 1}
    \cdots \sum_{k_{d - 1} = 0}^{k_{d - 2} + 1} x_1^{k_1} \cdots x_{d -
    1}^{k_{d - 1}} (1 - x_d^{k_{d - 1} + 2})\\
    & = & \frac{A_{d - 1} (x_1, \ldots, x_{d - 1}) - x_d^2 A_{d - 1} (x_1,
    \ldots, x_{d - 2}, x_{d - 1} x_d)}{1 - x_d} .
  \end{eqnarray*}
  
\end{proof}

It is clear from this recursive description that $A_d (x_1, \ldots, x_d)$ is a
rational function in the variables $x_1, \ldots, x_d$. Moreover, we can
describe the denominators as follows.

\begin{lemma}
  $\label{lem:Ax:denom}$For all positive integers $d$,
  \begin{equation}
    A_d (x_1, \ldots, x_d) = \frac{\alpha_d (x_1, \ldots, x_d)}{\prod_{r =
    1}^d (1 - x_1 x_2 \cdots x_r) }, \label{eq:Ax:denom}
  \end{equation}
  where $\alpha_d (x_1, \ldots, x_d) \in \mathbb{Z} [x_1, \ldots, x_d]$ is a
  polynomial of degree $d - 1$ in each variable. Moreover, $\alpha_d (0,
  \ldots, 0) = 1$ and the coefficient of $(x_1 \cdots x_d)^{d - 1}$ is $(-
  1)^{d - 1}$.
\end{lemma}

\begin{proof}
  Since $A_1 (x_1) = 1 / (1 - x_1)$, the claim is true for $d = 1$. For the
  purpose of induction on $d$, we fix $d$ and assume that \eqref{eq:Ax:denom}
  has already been shown to hold for $d - 1$. In particular, we may assume
  that
  \begin{eqnarray*}
    A_{d - 1} (x_1, \ldots, x_{d - 1}) & = & \frac{\alpha_{d - 1} (x_1,
    \ldots, x_{d - 1})}{\prod_{r = 1}^{d - 1} (1 - x_1 x_2 \cdots x_r) },\\
    A_{d - 1} (x_1, \ldots, x_{d - 2}, x_{d - 1} x_d) & = & \frac{\alpha_{d -
    1} (x_1, \ldots, x_{d - 2}, x_{d - 1} x_d)}{(1 - x_1 x_2 \cdots x_{d - 1}
    x_d) \prod_{r = 1}^{d - 2} (1 - x_1 x_2 \cdots x_r) } .
  \end{eqnarray*}
  Note that the least common multiple of the two denominators on the
  right-hand sides is $\prod_{r = 1}^d (1 - x_1 x_2 \cdots x_r)$. By
  Lemma~\ref{lem:Ax:rec}, we therefore have
  \begin{eqnarray}
    A_d (x_1, \ldots, x_d) & = & \frac{A_{d - 1} (x_1, \ldots, x_{d - 1}) -
    x_d^2 A_{d - 1} (x_1, \ldots, x_{d - 2}, x_{d - 1} x_d)}{1 - x_d}
    \nonumber\\
    & = & \frac{\beta_d (x_1, \ldots, x_d)}{(1 - x_d) \prod_{r = 1}^d (1 -
    x_1 x_2 \cdots x_r) },  \label{eq:Ax:denom:2}
  \end{eqnarray}
  for some $\beta_d (x_1, \ldots, x_d) \in \mathbb{Z} [x_1, \ldots, x_d]$. On
  the other hand, observe that
  \begin{equation*}
    A_{d - 1} (x_1, \ldots, x_{d - 1}) - x_d^2 A_{d - 1} (x_1, \ldots, x_{d -
     2}, x_{d - 1} x_d)
  \end{equation*}
  vanishes if we set $x_d = 1$. It follows that $\beta_d (x_1, \ldots, x_d)$
  is a multiple of $1 - x_d$, and this cancels with the factor $1 - x_d$ in
  the denominator of \eqref{eq:Ax:denom:2}, resulting in \eqref{eq:Ax:denom}
  with $\alpha_d (x_1, \ldots, x_d) = \beta_d (x_1, \ldots, x_d) / (1 - x_d)
  \in \mathbb{Z} [x_1, \ldots, x_d]$.
  
  Since the constant term of the denominator of \eqref{eq:Ax:denom} is $1$, it
  follows from the definition \eqref{eq:Ax:def} that the constant term
  $\alpha_d (0, \ldots, 0)$ of the numerator is $1$ as well. Finally, to
  extract the numerator term of highest degree, we note that
  \begin{eqnarray*}
    \beta_d (x_1, \ldots, x_d) & = & (1 - x_1 x_2 \cdots x_d) \alpha_{d - 1}
    (x_1, \ldots, x_{d - 1})\\
    &  & - x_d^2 (1 - x_1 x_2 \cdots x_{d - 1}) \alpha_{d - 1} (x_1, \ldots,
    x_{d - 2}, x_{d - 1} x_d) .
  \end{eqnarray*}
  It then follows from the induction hypothesis that the term of highest
  degree in $\beta_d (x_1, \ldots, x_d)$ is
  \begin{equation*}
    - x_d^2 (- x_1 x_2 \cdots x_{d - 1}) (- 1)^{d - 2} (x_1 x_2 \cdots
     x_d)^{d - 2} = x_d (- 1)^{d - 2} (x_1 x_2 \cdots x_d)^{d - 1} .
  \end{equation*}
  Consequently, after dividing $\beta_d (x_1, \ldots, x_d)$ by the factor $1 -
  x_d$, we find that the term of highest degree in $\alpha_d (x_1, \ldots,
  x_d)$ is $(- 1)^{d - 1} (x_1 x_2 \cdots x_d)^{d - 1}$ as claimed.
\end{proof}

Since $A_d (q) = A_d (q, \ldots, q)$, we observe that Lemma~\ref{lem:Ax:denom}
implies Theorem~\ref{thm:Aq:denom} as a special case as shown below.

\begin{proof}[Proof of Theorem~\ref{thm:Aq:denom}]
  In the special case where we set all variables $x_1, \ldots, x_d$ to $q$,
  the product $\prod_{r = 1}^d (1 - x_1 x_2 \cdots x_r)$ becomes the
  $q$-Pochhammer $(q ; q)_d$. As such, Lemma~\ref{lem:Ax:denom} shows that
  \begin{equation*}
    A_d (q) = \frac{\alpha_d (q)}{(q ; q)_d}, \quad \alpha_d (q) = \alpha_d
     (q, \ldots, q) \in 1 + q\mathbb{Z} [q],
  \end{equation*}
  as claimed in \eqref{eq:Aq:denom}. Moreover, the fact that the term of
  highest degree in $\alpha_d (x_1, \ldots, x_d)$ is \ $(x_1 \cdots x_d)^{d -
  1}$ implies that the degree of $\alpha_d (q)$ is $(d - 1) d$.
\end{proof}

We conclude this section by establishing the following asymptotic result for
$a_d (n)$ which we will use in Section~\ref{sec:Rk:formulas} for determining
the asymptotics of $R_k (n)$ as $n \rightarrow \infty$.

\begin{lemma}
  Fix a positive integer $d$. As $n \rightarrowlim \infty$, we have
  \begin{equation}
    a_d (n) = \frac{1}{d!} \frac{n^{d - 1}}{(d - 1) !} + O (n^{d - 2}) .
    \label{eq:ad:asy:first}
  \end{equation}
\end{lemma}

\begin{proof}
  Let $p_d (n)$ denote the number of partitions of $n$ into at most $d$ parts.
  Since $a_d (n)$ counts weak compositions of $n$ into $d$ parts such that
  each part does not exceed its predecessor by more than $1$, we clearly have
  $p_d (n) \leq a_d (n)$. On the other hand, we know from the proof of
  Lemma~\ref{lem:F:gf} that $a_d (n)$ equals the number of partitions of $n +
  T_d$ into exactly $d$ parts with Durfee triangle of size $d$. This implies
  that $a_d (n) \leq p_d (n + T_d)$. Combined, we therefore have
  \begin{equation}
    p_d (n) \leq a_d (n) \leq p_d (n + T_d) .
    \label{eq:ad:pd:bounds}
  \end{equation}
  On the other hand, it is well-known that
  \begin{equation}
    p_d (n) = \frac{1}{d!} \frac{n^{d - 1}}{(d - 1) !} + O (n^{d - 2}) .
    \label{eq:pd:asy:first}
  \end{equation}
  We refer, for instance, to \cite[Part~1, Problem~27]{polya-szego-1} or
  \cite[Theorem~4.2.1]{ramirez-alfonsin-frobenius} where the result is
  attributed to Laguerre and Schur. The bounds \eqref{eq:ad:pd:bounds}
  combined with \eqref{eq:pd:asy:first} show that $a_d (n) \sim p_d (n)$ as $n
  \rightarrow \infty$. This proves \eqref{eq:ad:asy:first}.
  
  For completeness, we indicate a direct proof of \eqref{eq:pd:asy:first}. To
  that end, observe that the argument used in the proof of
  Theorem~\ref{thm:Rk:asy} also applies to the coefficients $c_d (n)$ of
  \begin{equation*}
    \frac{\gamma_d (q)}{(q ; q)_d} = \sum_{n = 0}^{\infty} c_d (n) q^n,
  \end{equation*}
  where $\gamma_d (q)$ is some polynomial in $q$, and shows that, as $n
  \rightarrow \infty$,
  \begin{equation}
    c_d (n) = \frac{\gamma_d (1)}{d!} \frac{n^{d - 1}}{(d - 1) !} + O (n^{d -
    2}) . \label{eq:cd:asy:first}
  \end{equation}
  Since the numbers $p_d (n)$ have the simple generating function
  \begin{equation}
    \sum_{n \geq 0} p_d (n) q^n = \frac{1}{(q ; q)_d}, \label{eq:pd:gf}
  \end{equation}
  the claimed asymptotics \eqref{eq:pd:asy:first} follows from
  \eqref{eq:cd:asy:first} upon setting $\gamma_d (q) = 1$.
\end{proof}

\begin{corollary}
  \label{cor:alphad:1}Let $\alpha_d (q) = A_d (q) (q ; q)_d$ as in
  \eqref{eq:Aq:denom}. Then, for all positive integers $d$, we have $\alpha_d
  (1) = 1$.
\end{corollary}

\begin{proof}
  Setting $\gamma_d (q) = \alpha_d (q)$ in \eqref{eq:cd:asy:first}, it follows
  that
  \begin{equation*}
    a_d (n) = \frac{\alpha_d (1)}{d!} \frac{n^{d - 1}}{(d - 1) !} + O (n^{d -
     2}) .
  \end{equation*}
  Comparison with \eqref{eq:ad:asy:first} then shows that $\alpha_d (1) = 1$.
\end{proof}

\section{The generating function \texorpdfstring{$\mathcal{F}_k (q)$}{Fk(q)}}\label{sec:Fk}

We now return to the generating function $\mathcal{F}_k (q) \in q^{T_k}
\mathbb{Z} [[q]]$. To simplify exposition, we consider the normalization
\begin{equation*}
  F_k (q) = q^{- T_k} \mathcal{F}_k (q) = \sum_{n = 0}^{\infty} f_k (n) q^n
   \in 1 + q\mathbb{Z} [[q]],
\end{equation*}
so that $R_k (n) = f_k (n - T_k)$. The values of the numbers $f_k (n)$ for
small $k$ and $n$ are recorded in the following table (cf.~{\tmtt{A325188}} in
\cite{sloane-oeis}):
\begin{equation*}
  \begin{array}{|c|c|c|c|c|c|c|c|c|c|c|c|}
     \hline
     k\backslash n & 0 & 1 & 2 & 3 & 4 & 5 & 6 & 7 & 8 & 9 & 10\\
     \hline
     1 & 1 & 2 & 2 & 2 & 2 & 2 & 2 & 2 & 2 & 2 & 2\\
     \hline
     2 & 1 & 3 & 5 & 8 & 9 & 12 & 13 & 16 & 17 & 20 & 21\\
     \hline
     3 & 1 & 4 & 8 & 15 & 23 & 32 & 43 & 54 & 67 & 82 & 97\\
     \hline
     4 & 1 & 5 & 12 & 24 & 42 & 66 & 98 & 135 & 181 & 233 & 298\\
     \hline
     5 & 1 & 6 & 17 & 37 & 70 & 118 & 189 & 282 & 402 & 552 & 736\\
     \hline
     6 & 1 & 7 & 23 & 55 & 112 & 201 & 337 & 533 & 801 & 1158 & 1617\\
     \hline
     7 & 1 & 8 & 30 & 79 & 173 & 331 & 581 & 959 & 1502 & 2262 & 3286\\
     \hline
     8 & 1 & 9 & 38 & 110 & 259 & 528 & 974 & 1676 & 2724 & 4241 & 6368\\
     \hline
   \end{array}
\end{equation*}
In the previous section, we observed that $A_d (q)$ is a rational function in
$q$. In light of Lemma~\ref{lem:F:gf}, it follows that $F_k (q)$, and hence
$\mathcal{F}_k (q)$, are rational functions as well. For instance, from the
expressions for $A_d (q)$ given in the previous section for $d \in \{ 1, 2, 3,
4 \}$ we readily find that $F_1 (q) = (1 + q) / (1 - q)$ as well as
\begin{eqnarray*}
  F_2 (q) & = & \frac{1 + 2 q + q^2 + q^3 - q^4}{(1 - q)^2 (1 + q)},\\
  F_3 (q) & = & \frac{1 + 2 q + q^2 + 2 q^3 - q^4 - q^6 - q^7 + q^8}{(1 - q)^3
  (1 + q + q^2)},\\
  F_4 (q) & = & \frac{1 + 4 q + 6 q^2 + 7 q^3 + 6 q^4 + 2 q^5 - 5 q^7 - 5 q^8
  - 5 q^9 + q^{11} + 3 q^{12} + 2 q^{13} - q^{16}}{(1 - q)^4 (1 + q)^2 (1 + q
  + q^2) (1 + q^2)}
\end{eqnarray*}
which match the formulas derived in \cite{sharan-rook} (also see the entries
{\tmtt{A325168}}, {\tmtt{A382682}}, {\tmtt{A384562}} in \cite{sloane-oeis}).
The denominators of the latter three rational functions are $(q ; q)_2$, $(q ;
q)_3 / (1 + q)$ and $(q ; q)_4$, respectively. If written as rational
functions with denominator $(q ; q)_k$, the numerator degrees of $F_k (q)$ are
$k^2$ for $k \in \{ 1, 2, 3, 4 \}$. The next result proves that this
observation holds in general. In Lemma~\ref{lem:Fk:denom:odd}, we then show
that the denominator can be reduced to $(q ; q)_k / (1 + q)$ for all odd $k
\geq 3$.

\begin{theorem}
  \label{thm:Fk:denom}For all positive integers $k$,
  \begin{equation}
    F_k (q) = \frac{\varphi_k (q)}{(q ; q)_k} \label{eq:Fk:denom}
  \end{equation}
  where $\varphi_k (q) \in 1 + q\mathbb{Z} [q]$ is a polynomial of degree
  $k^2$. Moreover, the coefficient of $q^{k^2}$ in $\varphi_k (q)$ is $(-
  1)^{k - 1}$.
\end{theorem}

\begin{proof}
  By Lemma~\ref{lem:F:gf} combined with Theorem~\ref{thm:Aq:denom}, we have
  \begin{equation}
    F_k (q) = \sum_{d = 0}^k q^d A_d (q) A_{k - d} (q) = \sum_{d = 0}^k q^d 
    \frac{\alpha_d (q)}{(q ; q)_d}  \frac{\alpha_{k - d} (q)}{(q ; q)_{k -
    d}}, \label{eq:F:gf:noT}
  \end{equation}
  for polynomials $\alpha_d (q) \in 1 + q\mathbb{Z} [q]$ of degree $(d - 1)
  d$. This true also in the case $d = 0$, for which we have $A_0 (q) = 1$ and,
  hence, $\alpha_0 (q) = 1$. Recall the well-known fact that the quotients
  \begin{equation}
    \binom{k}{d}_q \assign \frac{(q ; q)_k}{(q ; q)_d (q ; q)_{k - d}}
    \label{eq:qbinomial}
  \end{equation}
  are polynomials in $\mathbb{Z} [q]$, known as the Gaussian polynomials or
  $q$-binomial coefficients. We conclude from \eqref{eq:F:gf:noT} that $F_k
  (q) = \varphi_k (q) / (q ; q)_k$ where
  \begin{equation}
    \varphi_k (q) = \sum_{d = 0}^k \binom{k}{d}_q q^d \alpha_d (q) \alpha_{k -
    d} (q) \in 1 + q\mathbb{Z} [q] . \label{eq:phik:alpha}
  \end{equation}
  It follows readily from \eqref{eq:qpochhammer} that $(q ; q)_d$ has degree
  $d (d + 1) / 2$. Accordingly, the $q$-binomial coefficients in
  \eqref{eq:qbinomial} have degree $d (k - d)$. Using further that $\alpha_d
  (q)$ has degree $(d - 1) d$, standard calculations show that, in the sum for
  $\varphi_k (q)$, the summand corresponding to $d = k$ is the unique one
  contributing the term of highest degree in $q$. That summand simplifies to
  $q^k \alpha_k (q)$ since the involved $q$-binomial coefficient as well as
  $\alpha_0 (q)$ equal $1$. Therefore, the degree of $\varphi_k (q)$ equals
  the degree of $q^k \alpha_k (q)$ which is $k + (k - 1) k = k^2$. Moreover,
  by Theorem~\ref{thm:Aq:denom}, the coefficient of $q^{k^2}$ is $(- 1)^{k -
  1}$.
\end{proof}

While the denominators for $A_d (q)$ provided by Theorem~\ref{thm:Aq:denom}
appear best possible, numerical evidence suggests that, as we observed in the
case $F_3 (q)$ above, the denominators for $F_k (q)$ provided by
Theorem~\ref{thm:Fk:denom} can be improved by removing the factor $1 + q$
whenever $k \geq 3$ is odd.

\begin{lemma}
  \label{lem:Fk:denom:odd}Write $F_k (q) = \varphi_k (q) / (q ; q)_k$ as in
  \eqref{eq:Fk:denom}. If $k \geq 1$ is odd then
  \begin{equation}
    \varphi_k (q) \in (1 + q) \mathbb{Z} [q] . \label{eq:phik:1q}
  \end{equation}
\end{lemma}

\begin{proof}
  From \eqref{eq:phik:alpha}, we have
  \begin{equation}
    \varphi_k (q) = \sum_{d = 0}^k \binom{k}{d}_q q^d \alpha_d (q) \alpha_{k -
    d} (q) = q^k \sum_{d = 0}^k \binom{k}{d}_q q^{- d} \alpha_d (q) \alpha_{k
    - d} (q) \label{eq:phik:alpha:2}
  \end{equation}
  where, for the second equality, we replaced $d$ by $k - d$ in the summand
  and used the symmetry
  \begin{equation*}
    \binom{k}{k - d}_q = \binom{k}{d}_q
  \end{equation*}
  of the $q$-binomial coefficients. Upon setting $q = - 1$ in both of the sums
  in \eqref{eq:phik:alpha:2} and comparing the results, we find that
  \begin{equation*}
    \varphi_k (- 1) = (- 1)^k \varphi_k (- 1) .
  \end{equation*}
  If $k$ is odd then this implies that $\varphi_k (- 1) = 0$ which is
  equivalent to \eqref{eq:phik:1q}.
\end{proof}

Apart from this observation, numerical evidence suggests that
Theorem~\ref{thm:Fk:denom} is best possible. In other words, it appears that
the minimal denominator of $F_k (q)$ equals $(q ; q)_k$, if $k$ is even, and
equals $(q ; q)_k / (1 + q)$ if $k \geq 3$ is odd (since, for $k
\geq 2$, the factor $1 + q$ divides both $(q ; q)_k$ and $\varphi_k
(q)$). By explicitly calculating $F_k (q)$, we have confirmed this conjecture
for $k \leq 20$.

\begin{example}
  The first case not covered by the approach in \cite{sharan-rook} is $k =
  6$. In that case, in accordance with Theorem~\ref{thm:Fk:denom}, we find
  $F_6 (q) = \varphi_6 (q) / (q ; q)_6$ where $\varphi_6 (q)$ is the degree
  $6^2$ polynomial
  \begin{eqnarray*}
    \varphi_6 (q) & = & 1 + 6 q + 15 q^2 + 25 q^3 + 34 q^4 + 35 q^5 + 31 q^6 +
    20 q^7 - 19 q^9 - 39 q^{10}\\
    &  & - 48 q^{11} - 50 q^{12} - 36 q^{13} - 19 q^{14} + 13 q^{15} + 30
    q^{16} + 45 q^{17} + 42 q^{18}\\
    &  & + 28 q^{19} + 11 q^{20} - 8 q^{21} - 21 q^{22} - 24 q^{23} - 15
    q^{24} - 9 q^{25} + 2 q^{26} + 3 q^{27}\\
    &  & + 5 q^{28} + 3 q^{29} + 2 q^{30} + 2 q^{31} - q^{36} .
  \end{eqnarray*}
\end{example}

\section{Asymptotics and exact formulas for \texorpdfstring{$R_k (n)$}{Rk(n)}}\label{sec:Rk:formulas}

Recall that a sequence $a (n)$ is \emph{constant recursive}, or $C$-finite,
of order $r$ if there exist complex numbers $c_0, c_1, \ldots, c_{r - 1}$ such
that
\begin{equation}
  a (n + r) = c_{r - 1} a (n + r - 1) + \cdots + c_1 a (n + 1) + c_0 a (n)
  \label{eq:cfinite:rec}
\end{equation}
for all $n \geq n_0$. The corresponding polynomial
\begin{equation}
  P (x) = x^r - c_{r - 1} x^{r - 1} - \cdots - c_1 x - c_0 \label{eq:charpoly}
\end{equation}
is called the characteristic polynomial. We refer to \cite{epsw-rec} or
\cite[Chapter~4]{kauers-paule-ct} for introductions to constant recursive
sequences. It is well-known that constant recursive sequences are exactly
those whose generating function is rational. More precisely, the condition
that $a (n)$ satisfies a recurrence of the form \eqref{eq:cfinite:rec} for all
$n \geq n_0$ is equivalent to
\begin{equation}
  \sum_{n = 0}^{\infty} a (n) q^n = \frac{N (q)}{D (q)} \label{eq:cfinite:rat}
\end{equation}
for polynomials $N, D \in \mathbb{C} [q]$, and $\deg (N) < \deg (D) + n_0$.
Moreover, if $P (x)$ is the characteristic polynomial \eqref{eq:charpoly} of
the recurrence then we can choose
\begin{equation*}
  D (q) = q^r P (1 / q) = 1 - c_{r - 1} q - \cdots - c_1 q^{r - 1} - c_0 q^r
   .
\end{equation*}
It therefore follows from Theorem~\ref{thm:Fk:denom} that, for every fixed
positive integer $k$, the sequence $R_k (n)$ is constant recursive. More
precisely, we find the following:

\begin{lemma}
  \label{lem:Rk:cfinite}Fix a positive integer $k$. The sequence $R_k (n)$ is
  constant recursive and, for all $n > k^2$, satisfies a recurrence
  \eqref{eq:cfinite:rec} of order
  \begin{equation*}
    r = \left\{\begin{array}{ll}
       T_k - 1, & \text{if $k \geq 3$ is odd},\\
       T_k, & \text{otherwise} .
     \end{array}\right.
  \end{equation*}
\end{lemma}

\begin{proof}
  By Theorem~\ref{thm:Fk:denom},
  \begin{equation}
    \mathcal{F}_k (q) = \sum_{n = 0}^{\infty} R_k (n) q^n = q^{T_k}
    \frac{\varphi_k (q)}{(q ; q)_k} \label{eq:Rk:gf:rat}
  \end{equation}
  where $\varphi_k (q) \in 1 + q\mathbb{Z} [q]$ has degree $k^2$. As in
  \eqref{eq:cfinite:rat}, this implies that $R_k (n)$ is constant recursive.
  Moreover, since $(q ; q)_k$ has degree $1 + 2 + \cdots + k = T_k$, we find
  that we can choose $n_0 = k^2 + 1$ in the above discussion. The recurrence
  corresponding to the denominators $(q ; q)_k$ has order $T_k$. On the other
  hand, we have shown in Lemma~\ref{lem:Fk:denom:odd} that, for odd $k
  \geq 3$, the denominators can be reduced to $(q ; q)_k / (1 + q)$.
  Accordingly, for odd $k \geq 3$, the order of the recurrence can be
  reduced to $T_k - 1$.
\end{proof}

As in the discussion after Lemma~\ref{lem:Fk:denom:odd}, it appears that the
orders of the recurrences in Lemma~\ref{lem:Rk:cfinite} are minimal.

\begin{example}
  \label{eg:F3:rat}We have
  \begin{equation*}
    \mathcal{F}_3 (q) = \sum_{n = 1}^{\infty} R_3 (n) q^n = q^{T_3} F_3 (q) =
     \frac{q^6 (1 + 2 q + q^2 + 2 q^3 - q^4 - q^6 - q^7 + q^8)}{(1 - q)^3 (1 +
     q + q^2)} .
  \end{equation*}
  Let $P (q) = (1 - q)^3 (1 + q + q^2) = 1 - 2 q + q^2 - q^3 + 2 q^4 - q^5$.
  It therefore follows as observed for \eqref{eq:cfinite:rat} that the
  sequence $R_3 (n)$ satisfies the recurrence
  \begin{equation*}
    R_3 (n + 5) = 2 R_3 (n + 4) - R_3 (n + 3) + R_3 (n + 2) - 2 R_3 (n + 1) +
     R_3 (n)
  \end{equation*}
  of order $5 = T_3 - 1$ for all $n > 14 - 5 = 3^2$. This coincides with the
  conclusion in Corollary~3.3 of \cite{sharan-rook}.
\end{example}

\begin{example}
  It likewise follows from Theorem~\ref{thm:Aq:denom} that the coefficients
  $a_d (n)$ of $A_d (q)$ discussed in Section~\ref{sec:Ad} are constant
  recursive and, for all $n > (d - 1) d - T_d$ or, equivalently, $n \geq
  T_{d - 2}$, satisfy a recurrence \eqref{eq:cfinite:rec} of order $T_d$. This
  proves, in particular, the generating functions and recurrences that are
  presently listed as conjectured by Chai Wah Wu for the cases $k \in \{ 4, 5,
  6, 7 \}$ in the OEIS \cite{sloane-oeis} (sequences \texttt{A244240},
  {\tmtt{A244241}}, {\tmtt{A244242}}, {\tmtt{A244243}}).
\end{example}

Let $\lambda_1, \lambda_2, \ldots, \lambda_s \in \mathbb{C}$ be the roots,
with multiplicities $m_1, m_2, \ldots, m_s$, of the characteristic polynomial
\eqref{eq:charpoly}. Then the sequence $a (n)$ admits an exact formula of the
form
\begin{equation}
  a (n) = \sum_{j = 1}^s \sum_{r = 0}^{m_j - 1} c_{j, r} n^r \lambda_j^n,
  \label{eq:cfinite:exact}
\end{equation}
which is valid for all $n \geq n_0$ (where, as before, $n_0$ is such that
the recurrence \eqref{eq:cfinite:rec} holds for all $n \geq n_0$), and
where $c_{j, r} \in \mathbb{C}$ are constants. This can be seen, for instance,
from the generalized binomial expansion applied to the partial fraction
decomposition of \eqref{eq:cfinite:rat}.

In our case of $\mathcal{F}_k (q) = \sum_{n \geq 1} R_k (n) q^n$, it
follows from Theorem~\ref{thm:Fk:denom} that the characteristic roots are all
roots of unity. Therefore, as $n \rightarrow \infty$, the asymptotically
dominating contribution comes from the root with the highest multiplicity.
Pursuing this line of reasoning ultimately results in the asymptotics claimed
in Theorem~\ref{thm:Rk:asy:intro} which is restated here. This in particular
proves, as conjectured in \cite{sharan-rook}, that $R_k (n) \sim C_k n^{k -
1}$ as $n \rightarrow \infty$ for suitable constants $C_k$.

\begin{theorem}
  \label{thm:Rk:asy}Fix a positive integer $k$. As $n \rightarrowlim \infty$,
  we have
  \begin{equation}
    R_k (n) = \frac{2^k}{k!} \frac{n^{k - 1}}{(k - 1) !} + O (n^{k - 2}) .
    \label{eq:Rk:asy:first}
  \end{equation}
\end{theorem}

\begin{proof}
  By Theorem~\ref{thm:Fk:denom}, we have
  \begin{equation*}
    \mathcal{F}_k (q) = \sum_{n = 1}^{\infty} R_k (n) q^n = q^{T_k} 
     \frac{\varphi_k (q)}{(q ; q)_k}
  \end{equation*}
  with $\varphi_k (q) \in 1 + q\mathbb{Z} [q]$. We can see from
  \eqref{eq:qpochhammer} that the roots of $(q ; q)_k$ are the roots of unity
  $\zeta$ of order $m$ with $m \leq k$ and their multiplicity is $\lfloor
  k / m \rfloor$. In other words,
  \begin{equation*}
    (q ; q)_k = (- 1)^k \prod_{m = 1}^k \Phi_m^{\lfloor k / m \rfloor} (q)
  \end{equation*}
  where $\Phi_m$ is the $m^{\text{th}}$ cyclotomic polynomial (the monic
  polynomial whose roots are precisely the $m^{\text{th}}$ primitive roots of
  unity). We write $\lambda_1 = 1$ and let $\lambda_2, \lambda_3, \ldots,
  \lambda_s$ be the other roots of unity of order up to $k$. We denote with
  $m_1, m_2, \ldots$ the corresponding multiplicities. Note that $m_1 = k$ and
  that $m_j \leq k / 2$ for $j > 1$. The exact formula
  \eqref{eq:cfinite:exact} therefore implies
  \begin{equation}
    R_k (n) = \sum_{r = 0}^{k - 1} c_{1, r} n^r + \sum_{j = 2}^s \sum_{r =
    0}^{m_j - 1} c_{j, r} n^r \lambda_j^n = \sum_{r = 0}^{k - 1} c_{1, r} n^r
    + O (n^{\lfloor k / 2 \rfloor - 1}) . \label{eq:Rk:asy:1}
  \end{equation}
  It remains to determine the leading coefficient $c_{1, k - 1}$. To this end,
  we write
  \begin{equation*}
    \mathcal{F}_k (q) = q^{T_k}  \frac{\varphi_k (q)}{(q ; q)_k} =
     \frac{\tilde{N} (q)}{(1 - q)^k \tilde{D} (q)}
  \end{equation*}
  for polynomials $\tilde{N} (q), \tilde{D} (q) \in \mathbb{Z} [q]$ with
  $\tilde{N} (1) \neq 0$ and $\tilde{D} (1) \neq 0$. Note that $\tilde{N} (q)
  = q^{T_k} \varphi_k (q)$ and $\tilde{D} (q) = (q ; q)_k / (1 - q)^k$.
  Performing part of the partial fraction expansion, we then have
  \begin{equation*}
    \mathcal{F}_k (q) = \frac{\gamma_k}{(1 - q)^k} + \frac{\gamma_{k - 1}}{(1
     - q)^{k - 1}} + \cdots + \frac{\gamma_1}{1 - q} + \frac{M (q)}{\tilde{D}
     (q)}
  \end{equation*}
  where $\gamma_j$ are constants and $M (q) \in \mathbb{Z} [q]$. Using
  \begin{equation*}
    \frac{(j - 1) !}{(1 - q)^j} = \sum_{n \geq 0} (n + j - 1) (n + j -
     2) \cdots (n + 1) q^n,
  \end{equation*}
  the coefficients $c_{1, r}$ in \eqref{eq:Rk:asy:1} can be obtained from the
  coefficients $\gamma_j$. In particular, the leading coefficient can be
  computed as
  \begin{equation*}
    c_{1, k - 1} = \frac{\gamma_k}{(k - 1) !} = \frac{1}{(k - 1) !} 
     \frac{\tilde{N} (1)}{\tilde{D} (1)} .
  \end{equation*}
  Since $\tilde{N} (q) = q^{T_k} \varphi_k (q)$, we have $\tilde{N} (1) =
  \varphi_k (1)$. On the other hand, we find that $\tilde{D} (1) = k!$ because
  \begin{equation*}
    \tilde{D} (q) = \frac{(q ; q)_k}{(1 - q)^k} = \prod_{r = 1}^k \frac{1 -
     q^r}{1 - q} = \prod_{r = 1}^k (1 + q + q^2 + \cdots + q^{r - 1}) .
  \end{equation*}
  Combined, we obtain
  \begin{equation*}
    c_{1, k - 1} = \frac{1}{(k - 1) !}  \frac{\tilde{N} (1)}{\tilde{D} (1)} =
     \frac{\varphi_k (1)}{(k - 1) !k!} .
  \end{equation*}
  By \eqref{eq:Rk:asy:1}, this proves the claimed asymptotics
  \eqref{eq:Rk:asy:first} if we can show that $\varphi_k (1) = 2^k$. Indeed,
  we can conclude from $\alpha_d (1) = 1$, established in
  Corollary~\ref{cor:alphad:1}, combined with \eqref{eq:phik:alpha} that
  \begin{equation*}
    \varphi_k (1) = \sum_{d = 0}^k \binom{k}{d} \alpha_d (1) \alpha_{k - d}
     (1) = \sum_{d = 0}^k \binom{k}{d} = 2^k,
  \end{equation*}
  as desired.
\end{proof}

\begin{remark}
  Note that we obtain the same leading asymptotics for the coefficients $f_k
  (n)$ of $F_k (q) = q^{- T_k} \mathcal{F}_k (q)$. Namely,
  \begin{equation*}
    f_k (n) = \frac{2^k}{k!} \frac{n^{k - 1}}{(k - 1) !} + O (n^{k - 2}) .
  \end{equation*}
\end{remark}

Recall that a sequence $f (n)$ is a quasi-polynomial of degree $m$ and
quasi-period $\delta$ if it can be expressed as
\begin{equation*}
  f (n) = \varphi_m (n) n^m + \varphi_{m - 1} (n) n^{m - 1} + \cdots +
   \varphi_0 (n)
\end{equation*}
where each $\varphi_j (n)$ has period $\delta$ (that is, $\varphi_j (n +
\delta) = \varphi_j (n)$). See, for instance, Section~4.4 in
\cite{stanley-ec1}. Equivalently, $f (n)$ is a quasi-polynomial of degree
$m$ and quasi-period $\delta$ if there exist polynomials $f_0 (n), f_1 (n),
\ldots, f_{\delta - 1} (n)$ such that $f (n) \equiv f_j (n)$ if $n \equiv j
\pmod{\delta}$.

Suppose that each root $\lambda_j$ of the characteristic polynomial
\eqref{eq:charpoly} is a root of unity of order $\delta_j$ (so that
$\lambda_j^{\delta_j} = 1$). Let $\delta$ be the least common multiple of
$\delta_1, \delta_2, \ldots, \delta_s$, and let $m$ be the maximum of $m_1,
m_2, \ldots, m_s$. Then $a (n)$ is a quasi-polynomial of degree $m - 1$ and
quasi-period $\delta$. Indeed, it follows from \eqref{eq:cfinite:exact} that,
for each $\nu \in \{ 0, 1, \ldots, \delta - 1 \}$, we have
\begin{equation}
  a (\delta n + \nu) = \sum_{j = 1}^s \sum_{r = 0}^{m_j - 1} c_{j, r} (\delta
  n + \nu)^r \lambda_j^{\nu}, \label{eq:cfinite:exact:qp}
\end{equation}
providing an explicit formula for $a (\delta n + \nu)$, for all $n \geq
n_0$, as a polynomial in $n$.

\begin{example}
  \label{eg:R3:qp}As detailed in \cite{sharan-rook}, James Sellers pointed
  out that it follows from the rational generating function for $R_3 (n)$
  given in Example~\ref{eg:F3:rat} that
  \begin{eqnarray*}
    R_3 (n) & = & \left\{\begin{array}{ll}
      6 m^2 - 15 m + 7, & \text{if $n = 3 m$,}\\
      6 m^2 - 11 m + 2, & \text{if $n = 3 m + 1$,}\\
      6 m^2 - 7 m - 1, & \text{if $n = 3 m + 2$,}
    \end{array}\right.
  \end{eqnarray*}
  provided that $n > 9$. We note that the right-hand side can be equivalently
  expressed as
  \begin{equation*}
    R_3 (n) = \frac{2}{3} n^2 - 5 n + \frac{59}{9} + \frac{4}{9} \cos \left(\frac{2 n \pi}{3} \right) .
  \end{equation*}
  In particular, the corresponding first-order asymptotics $R_3 (n) \sim
  \frac{2}{3} n^2$ matches the conclusion of Theorem~\ref{thm:Rk:asy} where it
  is shown that $R_k (n) \sim \frac{2^k}{k!} \frac{n^{k - 1}}{(k - 1) !}$.
  Likewise, in the special case $k = 4$, we obtain $R_4 (n) \sim \frac{1}{9}
  n^3$ as shown in \cite{sharan-rook}.
\end{example}

Representations as in the previous example can be worked out for $R_k (n)$ for
all fixed $k$, as a consequence of Theorem~\ref{thm:Fk:denom}.

\begin{corollary}
  \label{cor:Rk:qp}Let $k$ be a positive integer. Then there exists a
  quasi-polynomial $Q_k (n)$ of degree $k - 1$ and quasi-period $\operatorname{lcm}
  (1, 2, \ldots, k)$ such that $R_k (n) = Q_k (n)$ for all $n > k^2$. For $k =
  3$, the quasi-period is further reduced to $3$.
\end{corollary}

\begin{proof}
  Recall that the roots of the $q$-Pochhammer $(q ; q)_k$ are all the roots of
  unity of order up to $k$. More precisely, if $\mu_m$ denotes the primitive
  $m^{\text{th}}$ roots of unity, then each $\zeta \in \mu_m$ is a root of $(q
  ; q)_k$ of multiplicity $\lfloor k / m \rfloor$. Accordingly, it follows
  from Lemma~\ref{lem:Rk:cfinite} and the above discussion that $R_k (n)$ has
  a representation~\eqref{eq:cfinite:exact} which here takes the form
  \begin{equation}
    R_k (n) = \sum_{m = 1}^k \sum_{\zeta \in \mu_m} \sum_{r = 0}^{\lfloor k /
    m \rfloor - 1} c_{j, r} n^r \zeta^n \label{eq:Rk:exact}
  \end{equation}
  and is valid for all $n > k^2$. In particular, as for
  \eqref{eq:cfinite:exact:qp}, the right-hand side of \eqref{eq:Rk:exact} is a
  quasi-polynomial of degree $k - 1$ and quasi-period $\operatorname{lcm} (1, 2,
  \ldots, k)$.
  
  The reduction in Lemma~\ref{lem:Fk:denom:odd} implies that the maximum value
  of $r$ for $m = 2$ in \eqref{eq:Rk:exact} gets reduced by $1$. This lowers
  the quasi-period for $k = 3$ from $\operatorname{lcm} (2, 3) = 6$ to $3$ but does
  not affect the quasi-period for other values of $k$.
\end{proof}

As observed by Andrews \cite{andrews-pn-q-mod}, quasi-polynomials can always
be expressed via a single polynomial formula if we permit use of the floor
function. For instance, the piecewise formula in Example~\ref{eg:R3:qp} is
equivalent to the alternative representation
\begin{equation*}
  R_3 (n) = \frac{1}{3} (2 n - 3) (n - 7) + \frac{2}{3} \left\lfloor \frac{n
   + 1}{3} \right\rfloor + \frac{4}{3} \left\lfloor \frac{n}{3} \right\rfloor
\end{equation*}
for $n > 9$. The process of using computer algebra to automatically obtain
quasi-polynomial representations from an appropriate generating function is
described by Sills and Zeilberger \cite{sz-pmn-quasi}. The discussion in
\cite{sz-pmn-quasi} is focused on the examples of $p_m (n)$, the number of
partitions of $n$ into at most $m$ parts, as well as $D_k (n)$, the number of
partitions of $n$ whose Durfee square has size $k$. The approach, however,
applies equally to the numbers $R_k (n)$ in light of the generating function
provided by Theorem~\ref{thm:Fk:denom}.

\section{Conclusions and future work}

In Theorem~\ref{thm:Fk:denom} we showed that the generating function
$\mathcal{F}_k (q)$ of the number $R_k (n)$ of partitions of $n$ with Durfee
triangle of fixed size $k$ takes the form $q^{T_k} \varphi_k (q) / (q ; q)_k$.
This is derived from the $q$-multisum representation \eqref{eq:F:gf}. Neither
of these is as explicit as one would like for certain applications such as
extracting the asymptotic expansions of $R_k (n)$, as $n \rightarrow \infty$,
to higher order. For instance, computing the leading order asymptotics in
Theorem~\ref{thm:Rk:asy} required showing that the numerator polynomials
$\varphi_k (q)$ satisfy $\varphi_k (1) = 2^k$. It would be desirable to
determine the polynomials $\varphi_k (q)$ more explicitly.

By the relationship \eqref{eq:phik:alpha}, information on $\varphi_k (q)$ can
be inferred from corresponding information on the simpler polynomials
$\alpha_d (q)$ defined by \eqref{eq:Aq:denom}. The mentioned evaluations of
$\varphi_k (1)$, for instance, follow from the simpler evaluations $\alpha_d
(1) = 1$ that we showed in Corollary~\ref{cor:alphad:1}. More generally, it
appears worthwhile to investigate the values of the polynomials $\alpha_d (q)$
and $\varphi_k (q)$ at roots of unity. For instance, it appears that $\alpha_d
(- 1) = (- 1)^{\lfloor d / 2 \rfloor}$ for all integers $d \geq 0$. By
\eqref{eq:phik:alpha}, this implies the corresponding values
\begin{equation*}
  \varphi_k (- 1) = \left\{\begin{array}{ll}
     (- 2)^{k / 2}, & \text{if $k$ even,}\\
     0, & \operatorname{otherwise} .
   \end{array}\right.
\end{equation*}
The fact that $\varphi_k (- 1) = 0$ for odd $k$ follows from
Lemma~\ref{lem:Fk:denom:odd} according to which $1 + q$ divides $\varphi_k
(q)$ for odd $k$. Also, apart from these cases and the sporadic case
\begin{equation*}
  \alpha_4 (q) = (1 + q - q^4) (1 + 2 q - 2 q^4 - q^5 + q^8),
\end{equation*}
the polynomials $\alpha_d (q)$ and $\varphi_k (q)$ appear to be irreducible.

In a related direction, it would be of interest to obtain formulas for the
coefficients that arise in the (ordinary) partial fraction decomposition of
$\mathcal{F}_k (q)$ or its $q$-partial fraction decomposition in the sense of
Munagi \cite{munagi-qpartial} or, possibly, a partial fraction-type
expansion in the spirit of MacMahon \cite{sills-macmahon}. In the case of
$p_d (n)$, Rademacher \cite[Section~130]{rademacher1} considered the partial
fraction decomposition
\begin{equation*}
  \frac{1}{(q ; q)_d} = \sum_{k = 1}^d \sum_{\substack{
     0 \leq h < k\\
     (h, k) = 1
   }} \sum_{\ell = 1}^{\lfloor d / k \rfloor} \frac{C_{h, k, \ell}
   (d)}{(q - e^{2 \pi i h / k})^{\ell}}
\end{equation*}
and asked for formulas for the coefficients $C_{h, k, \ell} (d)$. Andrews
\cite{andrews-pn-q-mod} provided first instances of such formulas but they
are involved and not convenient for computations. Sills and Zeilberger
\cite{sz-rademacher} showed that certain coefficients $C_{h, k, \ell} (d)$
can be efficiently computed in a recursive manner. Rademacher conjectured that
the coefficients $C_{h, k, \ell} (d)$ converge as $d \rightarrow \infty$. This
was proved false by Drmota and Gerhold \cite{dg-rademacher} and,
independently, O'Sullivan \cite{osullivan-rademacher} after Sills and
Zeilberger \cite{sz-rademacher} provided strong numerical evidence that
Rademacher's conjecture is false. On the other hand, it remains an open
problem to clarify the relationship between the partial fraction
decompositions of $1 / (q ; q)_d$ and the limiting case $1 / (q ; q)_{\infty}$
which admits a version of a partial fraction decomposition by the
Hardy--Ramanujan--Rademacher formula for $p (n)$. For instance, as asked in
\cite{dg-rademacher}, do the coefficients $C_{h, k, \ell} (d)$ converge in a
weaker generalized sense as $d \rightarrow \infty$? These questions may also
be considered if one replaces (ordinary) partial fraction decompositions with
the $q$-partial fraction decomposition in the sense of Munagi
\cite{munagi-qpartial}. In the special case $h = k = 1$, Munagi
\cite{munagi-rademacher} proves that the analog of Rademacher's conjecture
actually holds true. It would be valuable to extend this analysis to other
coefficients as well as to similarly investigate the generating functions
studied in this paper.

Recall that every constant recursive integer sequence is necessarily
eventually periodic when reduced modulo an integer $M > 1$.
Corollary~\ref{cor:Rk:qp} allows us to bound the corresponding periods for
$R_k (n)$ modulo $M$. Namely, let $Q$ be the associated quasi-period (by
Corollary~\ref{cor:Rk:qp} we can always choose $Q = \operatorname{lcm} (1, 2, \ldots,
k)$; as well as $Q = 3$ if $k = 3$). Then we find
\begin{equation*}
  R_k (n + M Q) \equiv R_k (n) \pmod{M}
\end{equation*}
for all $n > k^2$, provided that $M$ is such that, for all $r \in \{ 0, 1,
\ldots, Q - 1 \}$, the values $R_k (m Q + r)$ are produced by polynomials in
$m$ with coefficients whose denominators are coprime to $M$ (these polynomials
necessarily have rational coefficients because they are integer-valued).

\begin{example}
  For instance, in the case $k = 3$, we find that $R_3 (n + 3 M) \equiv R_3
  (n)$ modulo $M$ for all $n > 9$. Likewise, $R_4 (n + 12 M) \equiv R_4 (n)$
  modulo $M$ for all $n > 16$ as well as $R_5 (n + 60 M) \equiv R_5 (n)$
  modulo $M$ for all $n > 25$. In the cases $k = 3$ and $k = 4$, these
  observations are also made in \cite{sharan-rook}. We note that in the
  cases $k \leq 5$ no restriction on $M$ is needed. On the other hand,
  for instance for $k = 6$, the congruences $R_6 (n + 60 M) \equiv R_6 (n)$
  modulo $M$ only hold for all $n > 36$ provided that $M$ is coprime to $3$.
\end{example}

Remarkably, much stronger congruences appear to hold in the case $M = 2$. As
observed in \cite{sharan-rook}, for $n > k^2$, the sequences $R_k (n)$
modulo $2$ are periodic with period $2$ if $k \in \{ 2, 3 \}$ and period $8$
if $k \in \{ 4, 5 \}$. Likewise, by explicitly computing the quasi-polynomial
representation in Corollary~\ref{cor:Rk:qp}, we find that the corresponding
period is $24$ if $k \in \{ 6, 7 \}$, $48$ if $k \in \{ 8, 9 \}$, and $480$ if
$k \in \{ 10, 11 \}$. It is natural to wonder whether the sequences $R_{2 k}
(n)$ and $R_{2 k + 1} (n)$ modulo $2$ always have the same (eventual) period.
We leave this question for future work. More generally, it would be of
interest to investigate further congruential properties of the numbers $R_k
(n)$, such as analogs of the Ramanujan congruences for the partition function.
Such investigations have been initiated and developed by Kronholm
\cite{kronholm-pnm}, \cite{kronholm-pnm-x} for the related numbers $p (n,
k)$ of partitions of $n$ into exactly $k$ parts. For recent subsequent work in
this direction, we refer to \cite{ekl-cranks-pnm} and the references
therein.

Finally, we wonder if it is fruitful to investigate the bivariate generating
function
\begin{equation*}
  \boldsymbol{F} (q, z) = \sum_{k = 0}^{\infty} F_k (q) z^k = \sum_{k =
   0}^{\infty} \sum_{n = 0}^{\infty} f_k (n) q^n z^k
\end{equation*}
in order to increase our understanding of the underlying numbers $R_k (n) =
f_k (n - T_k)$. In this direction, we note that, if we similarly write
\begin{equation*}
  \boldsymbol{A} (q, z) = \sum_{d = 0}^{\infty} A_d (q) z^d,
\end{equation*}
then Lemma \ref{lem:F:gf} becomes equivalent to the simple relationship
\begin{equation*}
  \boldsymbol{F} (q, z) =\boldsymbol{A} (q, q z) \boldsymbol{A} (q, z) .
\end{equation*}

\end{document}